\documentclass{article}

\usepackage{amsmath}
\usepackage{amssymb}
\usepackage{amsthm}
\usepackage[latin1]{inputenc}

\usepackage{graphicx}
\usepackage{epsfig}
\usepackage{epstopdf}
\usepackage{caption}
\usepackage{subcaption}
\usepackage{color}

\usepackage{ifthen}

\renewcommand{\div}{\operatorname{div}}

\newcommand{\Tt}{{\mathbb{T}}}

 \newcommand{\Rr}{\mathbb R}

 \newcommand{\af}{\alpha}
 
 \newcommand{\ep}{\epsilon}
\newcommand{\be}{\beta}
 \newcommand{\ga}{\gamma}
 
 \newcommand{\de}{\delta}
 
  \newcommand{\lam}{\lambda}
 \newcommand{\te}{\theta}

\renewcommand{\div}{\operatorname{div}}

\newcommand{\tr}{\operatorname{Tr}}

\newtheorem{teo}{Theorem}[section]

\newtheorem{Lemma}{Lemma}[section]
\newtheorem{Corollary}{Corollary}[section]
\newtheorem{Proposition}{Proposition}[section]

\newtheorem{Assumption}{A}

\begin{document}

\title{Time dependent mean-field games in the superquadratic case}
\author{Diogo
  A. Gomes\footnote{King Abdullah University of Science and Technology (KAUST), CEMSE Division and
  KAUST SRI, Uncertainty Quantification Center in Computational Science and Engineering, Thuwal 23955-6900. Saudi Arabia. e-mail: diogo.gomes@kaust.edu.sa.}, 
  Edgard Pimentel
  \footnote{Instituto Nacional de Matem\'atica Pura e Aplicada, IMPA. Estrada Dona Castorina, 110, 22460-320 Rio de Janeiro-RJ,
  Brazil. e-mail: edgardap@impa.br.}, 
H\'ector S\'anchez-Morgado.  
\footnote{Instituto de Matem\'aticas, Universidad Nacional Aut\'onoma
  de M\'exico. e-mail: hector@matem.unam.mx.}}

\date{\today} 

\maketitle

\begin{abstract}
We investigate time-dependent mean-field games with
superquadratic Hamiltonians and a power dependence on the
measure.
Such problems pose substantial mathematical challenges as the key
techniques used in the subquadratic case do not extend to the
superquadratic setting.  
Because of the superquadratic structure of the Hamiltonian, Lipschitz
estimates for the solutions of the Hamilton-Jacobi equation are obtained in the present paper through a
novel set of techniques. These explore the parabolic nature of the
problem through the nonlinear adjoint method. Well-posedness is
proven by combining Lipschitz regularity for the Hamilton-Jacobi
equation with polynomial estimates for solutions of the Fokker-Planck
equation.
Existence of classical solutions can is then established under
 conditions depending only on the growth of the
Hamiltonian and the dimension. Our
results also add to the current understanding of superquadratic
Hamilton-Jacobi equations.  
\end{abstract}

\thanks{
D. Gomes was partially supported by
KAUST baseline funds, KAUST SRI, Uncertainty Quantification Center in Computational Science and Engineering, and CAMGSD-LARSys through FCT-Portugal.

E. Pimentel is financed by CNPq-Brazil, grant 401795/2013-6.}

\section{Introduction}

The theory of mean-field games comprises a set of tools and methods, which aim at investigating differential games involving a (very) large number of rational, indistinguishable, players. These were introduced in the independent works of Lasry and Lions \cite{ll1, ll2, ll3, ll4} and Huang, Caines and Malham\'e \cite{C1, C2}. Since then, an intense research activity has been carried out in this field, as several authors have considered a variety of related problems. These include numerical methods \cite{lst}, \cite{DY}, \cite{CDY}, applications in economics \cite{llg1}, \cite{GueantT} and environmental policy \cite{lst}, finite state problems \cite{GMS}, \cite{GMS2}, \cite{GF}, explicit models \cite{Ge}, \cite{NguyenHuang}, obstacle-type problems \cite{GPat}, congestion \cite{GMit}, extended mean-field games \cite{GPatVrt}, \cite{GVrt}, probabilistic methods \cite{Carmona1}, \cite{Carmona2}, long-time behavior \cite{CLLP}, \cite{Cd1} and weak solutions \cite{Cd2}, \cite{Porb}, \cite{porretta}, to name only a few. For additional results, see also the recent surveys \cite{llg2}, \cite{cardaliaguet}, \cite{achdou2013finite}, or \cite{GS} and the references therein, and the College de France lectures by P-L. Lions \cite{LCDF,LIMA}.
 
A model time dependent mean-field game problem is given by
\begin{equation}
\label{eq:smfg0}
\begin{cases}
-u_t+H(x, Du)=\Delta u +g(m)\\
m_t-\div(D_pH m)=\Delta m,
\end{cases}
\end{equation}equipped with the initial-terminal conditions:
\begin{equation}
\label{itvp}
\begin{cases}
u(x,T)=u_T(x)\\
m(x,0)=m_0(x).
\end{cases}
\end{equation}
In the above,  the terminal instant $T>0$ is fixed. To simplify the
presentation, we consider the spatially periodic problem. For that, let $\Tt^d$ be the $d$-dimensional torus, identified as usual with the set $[0,1]^d$. Then we regard $u$
and $m$ as real valued functions defined over $\Tt^d\times\left[0,T\right]$. A typical Hamiltonian $H$ and
nonlinearity $g$ satisfying the assumptions that will be detailed in Section \ref{ma} are: 
$$H(x,p)\,=\,a(x)\left(1+|p|^2\right)^\frac{2+\mu}{2}\,+\,V(x),$$
and $$g(z)\,=\,z^\alpha,$$where $0\leq\mu<1$ and
$a,\,V\in\mathcal{C}^\infty(\Tt^d)$, $a,\,V>0$, are given. 

A fundamental question about MFG systems regards the existence of solutions. In the stationary setting, the first result in this direction was obtained in \cite{ll1}. Smooth solutions were studied in \cite{GM} (see 
also \cite{GIMY} for a related problem), \cite{GPM1}, and \cite{GPatVrt}. In \cite{ll2} the authors  addressed 
for the first time the question of
existence of weak-solutions to \eqref{eq:smfg0}-\eqref{itvp}. The planning problem was investigated in \cite{porretta} and \cite{Porb}, also 
in the framework of weak solutions. In the quadratic Hamiltonians case, existence of smooth solutions has been established in \cite{CLLP}. We emphasize the fact that the proof in \cite{CLLP} relies on a Hopf-Cole transformation and does not seem to extend to more general cases behaving like $|p|^2$ at infinity. As presented in \cite{LIMA}, mean-field games with quadratic or subquadratic growth in the Hamiltonian, and the power nonlinearity $g(m)=m^\alpha$, have classical solutions under some bounds on $\alpha$.
In \cite{GPM2} the authors extended and improved substantially these results in the subquadratic setting. 
Also in the subquadratic setting, existence of smooth solutions was studied in \cite{GPim1}, in the whole
space, and in \cite{GPim2} for logarithmic nonlinearities. 

To the best of our knowledge, superquadratic time dependent mean-field games have not been studied in the literature before the present paper, nor can they be addressed by a minor extension of existent results. We stress the fact that previous arguments regarding existence of weak solutions do not extend to the superquadratic setting and, therefore, not even in this case the existence of solutions had been established.

Indeed, many of the key estimates for quadratic or subquadratic
mean field games are simply not valid
for superquadratic Hamiltonians.
For instance, the Gagliardo-Nirenberg estimates combined with the Crandall-Amann technique \cite{CrAm} are no longer valid due to the growth of the Hamiltonian.
Consequently,
in the superquadratic case, estimates for Hamilton-Jacobi equations are substantially more delicate and require arguments quite distinct from the ones used in the quadratic or subquadratic cases. See, for instance, the recent developments concerning H\"older estimates in \cite{Bsuper}, \cite{CDLP}, \cite{CarSil}. 
To show existence of smooth solutions for the case of superquadratic Hamiltonians, we develop in this paper a new class of Lipschitz estimates. These are proven by identifying additional regularizing effects, which combine the parabolic structure of the Hamilton-Jacobi equations
with its stochastic optimal control origin.
This is achieved by employing the nonlinear adjoint method \cite{E3} in a novel way.

Our main result is the following:
\begin{teo}\label{teo:intro2}
Assume that A\ref{ah}-A\ref{alpha3} from Section \ref{ma} hold. Then there exists a $C^\infty$ solution $(u,m)$ to \eqref{eq:smfg0} under the initial-terminal conditions \eqref{itvp}, with $m>0$.
\end{teo}

We observe that uniqueness of solutions to \eqref{eq:smfg0}-\eqref{itvp} follows from earlier results in \cite{ll1,ll2}.

The key assumptions A\ref{ah}-A\ref{alpha3} are discussed in Section \ref{subsec:assumptions}. An outline of the proof of this Theorem is described
in Section \ref{outline}. The various steps of the proof are detailed in the remaining Sections. In particular, in Section \ref{sec:sec12} we establish Lipschitz regularity for H-J equations (see Theorem \ref{uReg1c} stated in the next Section). 

The authors thank P. Cardaliaguet, P-L. Lions, A. Porretta and P. Souganidis for very useful comments and suggestions.

\section{Main assumptions and proof outline}
\label{ma}

We begin by discussing the main assumptions used in the present paper, and which cover a range of relevant problems.
This Section ends with the statement of the key Theorems and Lemmas, as well as an outline of the proof of Theorem \ref{teo:intro2}.

\subsection{Assumptions}\label{subsec:assumptions}

We assume our problem satisfies the following general hypotheses:
\begin{Assumption}
\label{ah}
The Hamiltonian
$H:\Tt^d\times \Rr^d \to \Rr$ is $C^\infty$ and 
\begin{enumerate}
\item 
For fixed $x$, the map $p\mapsto H(x,p)$ is strictly convex;  
\item 
Additionally, $H$ satisfies the coercivity condition
\[
\lim_{|p|\to \infty} \frac{H(x,p)}{|p|}=+\infty,   
\] 
and, without loss of generality, we require further that $H(x,p)\geq 1$. 
\end{enumerate}
\end{Assumption}

\begin{Assumption}
\label{ag} The function 
$g:\Rr_0^+\to \Rr$ is non-negative and increasing. 

Finally, $u_0, m_0\in C^\infty(\Tt^d)$ with $m_0\geq 0$ and $\int_{\Tt^d} m_0=1$. 
\end{Assumption}

Since $g$ is increasing and non-negative, it follows that there exists a
convex increasing function $G:\Rr^+_0\to\Rr$ such that $g(z)=G'(z)$.   

The Legendre transform of $H$ is given by
$L(x,v)=\sup\limits_p \left( -p\cdot v-H(x,p)\right)$.  
Then, if we set
\begin{equation}
\label{ele}
\hat L(x,p)=D_pH(x,p) p-H(x,p),  
\end{equation}
by standard properties of the Legendre transform
$
\hat L(x,p)=L(x,-D_pH(x,p)).
$

\begin{Assumption}
\label{aele}
For some $c, C>0$
\[
\hat L(x,p)\geq c H(x,p)-C.
\]
\end{Assumption}

For convenience and definiteness, we choose $g$ to be a power nonlinearity. 

\begin{Assumption}
\label{ag2}
$g(m)=m^\alpha$, for some $\alpha>0$.
\end{Assumption}

Our results can be generalized easily to the setting in which $g$ depends simultaneously on $m$ and $x$, provided appropriate conditions concerning the growth and the bounds of $g$ are assumed. This will not be pursued here to keep the presentation elementary.

\begin{Assumption}
\label{strong}
$H$
satisfies the following bounds
\[
|D_xH|, |D^2_{xx}H|\leq CH+C, 
\]
and, for any symmetric matrix $M$, and any $\delta>0$ there exists $C_\delta$ such that 
\[
\tr(D^2_{px} H M) \leq \delta \tr(D^2_{pp} H M^2)+C_{\delta} H. 
\]
\end{Assumption}

Because $H\geq 1$, the inequality in
the previous Assumption is equivalent to $|D_xH|, |D^2_{xx}H|\leq \tilde C H$, for
some constant $\tilde C$.  

\begin{Assumption}
\label{bcc}
$m_0\geq \kappa_0$
for some $\kappa_0\in \Rr^+$.
\end{Assumption}

The preceding hypotheses are the same as the corresponding ones in \cite{GPM2}.  
The next group of Assumptions
is distinct and encodes the
superquadratic nature of the Hamiltonian. 

\begin{Assumption}
\label{superq}
For some $0<\mu<1$, the Hamiltonian satisfies
\[
c_1|p|^{2+\mu}+C_1\leq H\leq c_2|p|^{2+\mu}+C_2,
\]where $c_i$ and $C_i$ are non-negative constants.
\end{Assumption}

\begin{Assumption}
\label{adph}
The following estimate holds
\[|D_pH(x,p)|^2\le C|p|^\mu H(x,p)+C.\]
\end{Assumption}

\begin{Assumption}
\label{aspdpph}
$H$ satisfies the following bounds
\[
\left|D^2_{xp}H\right|^2\leq CH^{\frac{2+2\mu}{2+\mu}},
\]
and, for any symmetric matrix $M$,
\[
\left|D^2_{pp}HM\right|^2\leq  C H^\frac{\mu}{2+\mu}\tr(D^2_{pp} H M M),  
\]
where $\mu$ and $C$ are given constants.
\end{Assumption}

Note that, in particular, the previous hypothesis implies that for
any function $u$, $H(x,Du)$ satisfies the following estimates:
\begin{equation}
\label{magic}
\big|\div(D_pH(x,Du))\big|^2\leq C H^\frac{\mu}{2+\mu}
\big(\tr D^2_{pp}HD^2uD^2u \big)+CH^\frac{2+2\mu}{2+\mu}. 
\end{equation}

\begin{Assumption}
\label{alpha3}
The exponent $\alpha$ satisfies $\alpha <\dfrac{2}{d(1+\mu)-2}.$
\end{Assumption}

\subsection{Outline of the proof}\label{outline}

The proof of Theorem \ref{teo:intro2} starts by considering a regularized version of \eqref{eq:smfg0}. It consists of
replacing $g(m)$ by: 
\begin{equation}\label{gep}
g_\epsilon(m)=\eta_\epsilon*g(\eta_\epsilon*m),
\end{equation} 
where $\eta_\epsilon$ is a standard, symmetric, mollifying kernel. This yields the regularized model:
\begin{equation}
\label{eq:smfg}
\begin{cases}
-u_t^\epsilon+H(x, Du^\epsilon)=\Delta u^\epsilon +g_\epsilon(m^\epsilon)\\
m_t^\epsilon-\div(D_pH m^\epsilon)=\Delta m^\epsilon.  
\end{cases}
\end{equation}
For convenience, we set $g_0=g$. The special structure of \eqref{gep} makes it possible to prove estimates for \eqref{eq:smfg} which are uniform in $\ep$. Existence of $C^\infty$ solutions for \eqref{eq:smfg}-\eqref{itvp} follows from standard arguments
using some of the ideas in \cite{cardaliaguet}, as detailed in \cite{PIM}.

The proof of Theorem \ref{teo:intro2} proceeds by considering
polynomial estimates for $g_\ep(m^\ep)$ in terms of $Du^\ep$ as stated
in the following Theorem: 

\begin{teo}\label{cor:cor2sec84}
Let $(u^\epsilon,m^\epsilon)$ be a solution of \eqref{eq:smfg}. Assume
A\ref{ah}-\ref{aspdpph} hold. 
Let $\theta>1$, $0\leq \upsilon\leq 1$.
For $\beta_0\in\left[1,\dfrac{d(1+\mu)}{d(1+\mu)-2}\right)$, let 
$\beta_{\upsilon,\te}=\dfrac{\theta\beta_0}{\theta+\upsilon-\theta\upsilon}$ 
and 
\begin{equation}\label{eq:eq3sec84}
r_\te=\frac{d(\theta-1)+2}{2}.
\end{equation}Suppose
\begin{equation}\label{eq:eq4sec84}
p_{\upsilon,\te}=\frac{\beta_{\upsilon,\te}}\alpha>1.
\end{equation}
Then, for $r=r_\te$ and $p=p_{\upsilon,\te}$, we have
\begin{equation}\label{capa}
\|g_\ep(m^\ep)\|_{L^\infty(0,T;L^p(\Tt^d))}\leq C+C\left\|Du^\ep\right\|_{L^\infty(0,T;L^\infty(\Tt^d))}^{\frac{(2+2\mu)r\upsilon\af}{\theta\be_0}}, 
\end{equation}
where $C$ is independent of $\epsilon$. 
\end{teo}Theorem \ref{cor:cor2sec84} is proven in Section
\ref{sec84}. Then we establish $L^\infty$ bounds for $u^\ep$ in terms
of $g_\ep(m^\ep)$, as in the following Lemma: 

\begin{Lemma}
\label{imphc0}
Suppose $(u^\epsilon,m^\epsilon)$ is a solution of \eqref{eq:smfg} and $H$ satisfies 
A\ref{ah}. Then, if $p>\frac{d}{2},$
\begin{equation}\label{capb}
\|u^\epsilon\|_{L^\infty(\Tt^d\times[0,T])}\leq C+
C\|g_\ep(m^\ep)\|_{L^\infty(0,T;L^{p}(\Tt^d))}, 
\end{equation}
where $C$ is independent of $\epsilon$. 
\end{Lemma}
The proof of Lemma \ref{imphc0} is presented in Section \ref{pub}. To
estimate $Du^\ep$ in terms of $g_\ep(m^\epsilon)$ we apply the nonlinear adjoint
method (see \cite{E3}), which yields the following estimate: 

\begin{teo}\label{uReg1c}
Suppose  A\ref{ah}-A\ref{alpha3} hold. 
Let $(u^\ep,m^\ep)$ be a solution of \eqref{eq:smfg} and assume that $p>d$.  
Then 
\begin{align}\nonumber
\|Du^\ep\|_{L^\infty(0,T;L^\infty(\Tt^d))}\leq &
C+C\|g_\ep(m^\ep)\|^\frac{1}{1-\mu}_{L^\infty(0,T;L^p(\Tt^d))}\\\label{capc}
&+C\|g_\ep(m^\ep)\|^\frac{1}{1-\mu}_{L^\infty(0,T;L^p(\Tt^d))}\|u\|_{L^\infty(0,T;L^\infty(\Tt^d))}^\frac{1}{1-\mu}, 
\end{align}
where $C$ is independent of $\epsilon$. 
\end{teo}
This Theorem is established in Section \ref{ram}. To prove Theorem \ref{teo:intro2}
we combine the estimates in Theorem \ref{cor:cor2sec84}, Lemma
\ref{imphc0} and Theorem \ref{uReg1c}, obtaining Lipschitz regularity
for $u^\ep$. This is done in Section \ref{sec:sec12}. It follows from \eqref{capa}, combined with 
\eqref{capc}, that $$\|Du^\ep\|_{L^\infty(\Tt^d\times[0,T])}\leq C+C\|Du^\ep\|_{L^\infty(\Tt^d\times[0,T])}^\zeta,$$
where, if $\alpha$ is small enough, $\zeta<1$.
The precise bound for $\alpha$ is the one given in Assumption A\ref{alpha3}.
Lastly, we obtain Lipschitz regularity:

\begin{teo}\label{teo:sec12}
Let $(u^\epsilon,m^\epsilon)$ be a solution of \eqref{eq:smfg}-\eqref{itvp}. Suppose that
A\ref{ah}-\ref{alpha3} hold. 
Then, $Du^\epsilon\in L^\infty(\Tt^d\times\left[0,T\right])$, with bounds uniform in $\epsilon$.
\end{teo}

We now present the proof of Theorem \ref{teo:intro2}.

\begin{proof}[Proof of Theorem \ref{teo:intro2}.]

By theorem \ref{teo:sec12}, we have Lipschitz regularity for $u^\ep$, uniformly in $\ep$. 
Thus 
the growth of the Hamiltonian plays no role in further gain of regularity.
Then, a number of additional estimates can be derived, see \cite{GPM2}.
These ensure, in particular, that $u^\ep$ and $m^\ep$ are H\"older continuous, uniformly in $\epsilon$. 
Thus, through some subsequence we have that $u^\ep\to u$ in $\mathcal{C}^{0,\ga}(\Tt^d\times\left[0,T\right])$ and $m^\ep\to m$ in $\mathcal{C}^{0,\ga}(\Tt^d\times\left[0,T\right])$, as $\ep\to 0 $.
This shows that $u$ is a (viscosity) solution of the first equation in \eqref{eq:smfg0}.
Furthermore, additional bounds on $D^2u^\epsilon$ provide enough compactness to 
conclude that $m$ solves
\[
m_t-\div(D_pH(x,Du)m)=\Delta m,
\]
as a weak solution, 
i.e.,
$$\int_0^T\int_{\Tt^d}\left(-\phi_t+D_pHD\phi-\Delta\phi\right)mdxdt\,=\,0,$$f
or every $\phi\in\mathcal{C}^\infty_c(\Tt^d)$. 
By the results in \cite{GPM2}, we have uniform bounds in every Sobolev space for $(u^\ep,m^\ep)$.
Finally, observing that $(u,m)$ satisfies the same estimates as $(u^\ep,m^\ep)$, we obtain existence of smooth solutions. 
\end{proof}

The rest of this paper is organized as follows: the next Section presents some elementary estimates from \cite{GPM2}. In Section \ref{igain}, we obtain higher integrability for $m^\ep$, see Theorem \ref{emsupere}. The proof of Theorem \ref{cor:cor2sec84} is presented in Section \ref{sec84}. In Sections \ref{pub}-\ref{ram}, we establish Lemma \ref{imphc0} and Theorem \ref{uReg1c}. Lipschitz regularity for the Hamilton-Jacobi equation is established in Section \ref{sec:sec12}.

\section{Elementary estimates}\label{sec:eles}
\label{apes}

Next we recall several estimates for solutions of \eqref{eq:smfg}. These have appeared
(either in the present form or in related versions)
in \cite{ll1,ll2,CLLP,LIMA,E2}. For ease of presentation, we omit here the proofs
which can be found in \cite{GPM2}.
\begin{Proposition}[Stochastic Lax-Hopf estimate]
\label{plb}
Suppose A\ref{ah} holds. Let $(u^\epsilon,m^\epsilon)$ be a solution
to \eqref{eq:smfg}.  
Then, for any smooth vector field $b:\Tt^d\times(t,T)\to\Rr^d$,
and any solution to 
\begin{equation}
\label{diflaw2}
\zeta_s+\div(b\zeta)=\Delta \zeta, 
\end{equation}
with $\zeta(x,t)=\zeta_0$ 
we have
 the following upper bound: 
\begin{align}
\label{lhe2}
\int_{\Tt^d} u^\epsilon(x,t)\zeta_0(x) dx \le
&\int_t^T\int_{\Tt^d}\bigl(L(y,b(y,s))
+g_\epsilon(m^\epsilon)(y,s)\bigr)\zeta(y,s) dy ds\\\nonumber
&+\int_{\Tt^d} u^\epsilon_T(y) \zeta(y,T). 
\end{align} 
\end{Proposition}We notice that, for $b=-D_pH(x,Du)$, the inequality in \eqref{lhe2} is attained.

\begin{Proposition}[First-order estimate]
\label{pehm}
Assume A\ref{ah}-\ref{aele} hold. 
Let $(u^\epsilon ,m^\epsilon)$ be a solution of \eqref{eq:smfg}. Then
\begin{equation}
\label{ihm}
\int_0^T \int_{\Tt^d} c H(x, D_xu^\epsilon) m^\epsilon
+G(\eta_\epsilon*m^\epsilon) dx dt 
\leq CT+C \left\|u^\epsilon(\cdot, T)\right\|_{L^\infty(\Tt^d)}, 
\end{equation}
where $G'=g$.
\end{Proposition}

\begin{Proposition}[Second-order estimate]\label{elsoe1}
Assume A\ref{ah}-\ref{bcc}
hold. 
Let $(u^\epsilon,m^\epsilon)$ be a solution of \eqref{eq:smfg}.
\begin{align*}
\int_0^T\int_{\Tt^d} & g'(\eta_\epsilon*m^\epsilon)|D_x (\eta_\epsilon*m^\epsilon)|^2
+
\tr(D_{pp}^2H(D^2_{xx}u^\epsilon)^2) m^\epsilon
\leq C.
\end{align*}
\end{Proposition}

\begin{Corollary}
\label{mhjr}
Assume  A\ref{ah}-\ref{bcc} hold. 
Let $(u^\epsilon,m^\epsilon)$ be a solution of \eqref{eq:smfg}. Then
\[
\int_0^T \|\eta_\epsilon*m^\epsilon\|_{L^{\frac{2^*}2 (\alpha+1)}(\Tt^d)}^{\alpha +1}dt\leq C.
\]
\end{Corollary}

\section{Regularity for the Fokker-Planck equation}
\label{igain}

Next, building upon the second-order estimate of Proposition \ref{elsoe1}, we obtain improved integrability for $m^\epsilon$.
In Section \ref{elep}, the integrability of $m^\epsilon$ is controlled
in terms of $L^p$ norms of $D_pH(x,Du^\ep(x))$. 

In the superquadratic case, further arguments yield uniform estimates for $Du^\ep$ in $L^\infty(\Tt^d\times[0,T])$ rather than in $L^r(0,T;L^p(\Tt^d))$, which was the space used in the
subquadratic setting.

Along this Section,
the function $H$ and its derivatives will be evaluated at $(x,Du^\ep(x))$. However, to ease the notation, we will omit this argument.

\subsection{Regularity by the second-order estimate}
\label{rele}

We begin by addressing the regularity of the Fokker-Planck equation
by applying the second-order estimate from the previous section. 

\begin{teo}
\label{emsupere}
Assume A\ref{ah}-\ref{bcc} and A\ref{aspdpph} hold.
Let $(u^\epsilon,m^\epsilon)$ be a solution of \eqref{eq:smfg}. Then
for $d\geq 2$, $\|m^\epsilon(\cdot,t)\|_{L^\infty(0,T;L^r(\Tt^d))}$ is bounded for any 
$1\le r<\frac{d (1+\mu) }{d (1+\mu) - 2}$, uniformly in $\epsilon$.
\end{teo}
\begin{proof}
We will omit the $\epsilon$ to simplify the notation. 
We start by defining an increasing sequence $\be_n$ such 
that $\|m(\cdot, t)\|_{L^{1+\be_n}(\Tt^d)}$ is bounded.
We 
set $\be_0=0$ so that $\|m(\cdot, t)\|_{L^{1+\be_0}(\Tt^d)}=1\leq C$.
 
At this point, it is critical
to control 
$
\int_{\Tt^d} \div(D_pH)m^{\be+1}dx 
$.
This will be done using Assumption A\ref{aspdpph}. In fact, using
\eqref{magic} we have: 
\begin{align}\notag
&\int_{\Tt^d} \div(D_pH)m^{\be+1}dx\\\notag
&\leq 
\int_{\Tt^d} C H^\frac{\mu}{2(2+\mu)} m^{\frac{\mu}{2(2+\mu)}}
\tr \left(D^2_{pp}HD^2uD^2u \right)^{1/2} m^{1/2} m^{\beta+\frac{1}{(2+\mu)}}\\\notag
&\quad + \int_{\Tt^d}CH^\frac{1+\mu}{2+\mu} m^\frac{1+\mu}{2+\mu} m^{\beta+\frac{1}{(2+\mu)}}\\\label{key-supere}
&\leq C_\de\int_{\Tt^d} H m +C_\de\int_{\Tt^d} \tr\left(D^2_{pp}HD^2uD^2u \right)m+\de\int_{\Tt^d} m^{(2+\mu)\beta+1}.
\end{align}
We note that the time integral of the first two terms on the right-hand side of the previous 
inequalities is bounded by Propositions \ref{pehm} and \ref{elsoe1}. Because of the Sobolev's Theorem we proceed by examining the cases $d>2$ and $d=2$ separately.

Consider the case $d>2$.
Let $\be_{n+1}=\frac{2}{d(1+\mu)}(\be_n+1).$
Then $\be_n$ is the $n^{th}$ partial sum of a geometric series with
term $\dfrac{2^n}{d^n(1+\mu)^n}$. Therefore
$\lim\limits_{n\to \infty}\be_n=\dfrac{2}{d(1+\mu)-2}$.
We define $q_n=\frac{2^*}2(\be_{n+1}+1),$ where $2^*$ is the critical Sobolev exponent given by $$2^*=\frac{2d}{d-2}.$$
Hence we have $$\|m\|_{L^{(2+\mu)\be_{n+1}+1}(\Tt^d)}\leq \|m\|_{L^{1+\be_n}(\Tt^d)}^{1-\lam_n}\|m\|_{L^{q_n}(\Tt^d)}^{\lam_n},$$
where $\frac{\lambda_n}{q_n}+\frac{1-\lambda_n}{1+\be_n}=\frac{1}{(2+\mu)\be_{n+1}+1},$
and thus:
\begin{equation}\label{eq:exponente}
  \lambda_n=\frac{q_n}{q_n-\be_n-1}\frac{(2+\mu)\be_{n+1}-\be_n}{1+(2+\mu)\be_{n+1}}.
\end{equation}
Since $\| m \|_{L^{1+\be_n}(\Tt^d)}\leq C,$ we get
\begin{equation}
 \label{m-interpole}
\int_{\Tt^d} m^{(2+\mu)\be_{n+1} +1} dx
=\|m\|_{L^{(2+\mu)\be_{n+1}+1}(\Tt^d)}^{(2+\mu)\be_{n+1}+1}
\leq C\|m\|_{L^{q_n}(\Tt^d)}^{\lam_n ((2+\mu)\be_{n+1}+1)}.
\end{equation}

Setting $\be=\be_{n+1}$, from \eqref{key-supere} and \eqref{m-interpole}, 
we get for any $\tau\in[0,T]$ 
\begin{align}
&\label{igm-eq6ae} 
\int_{\Tt^d} m^{\be_{n+1}+1}(x,\tau)dx +\frac{4\be_{n+1}}{\be_{n+1}+1}
\int_0^\tau\int_{\Tt^d}|D_xm^{\frac{\be_{n+1}+1}{2}}(x,t)|^2dx\,dt
\notag \\
&=  \int_{\Tt^d} m^{\be_{n+1}+1}(x,0)dx + \be
\int_0^\tau \int_{\Tt^d} \div(D_pH)m^{\be_{n+1}+1}dx dt\notag \\
&\leq  \int m^{\be_{n+1}+1}(x,0)dx + C_\de \int_0^\tau\int_{\Tt^d} H m \\\notag
&+ C_\de \int_0^\tau\int_{\Tt^d}\left(\tr D^2_{pp}HD^2uD^2u \right)m
+\delta \int_0^\tau\|m\|_{L^{q_n}(\Tt^d)}^{\lam_n ((2+\mu)\be_{n+1}+1)} dt.
\end{align}
From the definition of $\be_{n+1}$ it follows that $\lam_n ((2+\mu)\be_{n+1}+1)=1+\beta_{n+1}.$
Hence,
\begin{align}\label{igm-eq3ae}
\|m\|_{L^{q_n}(\Tt^d)}^{\lam_n ((2+\mu)\be_{n+1}+1)}&=
\|m^{\frac{\be_{n+1}+1}{2}}\|_{L^{2^*}(\Tt^d)}^2\\\notag
&\le C+C\left(\int_{\Tt^d} |D_xm^{\frac{\be_{n+1}+1}{2}}|^2 +\int_{\Tt^d} m^{\be_{n+1}+1}\right).
 \end{align}
Using elementary inequalities and $\int mdx=1$ we have
for any $\zeta>0$ that: 
$\int_{\Tt^d} m^{\be_{n+1}+1}
\le  C_\zeta+\zeta\|m\|_{L^{q_n}(\Tt^d)}^{\beta_{n+1}+1}$.
Thus it follows
\[\|m\|_{L^{q_n}(\Tt^d)}^{\beta_{n+1}+1}\le C_\zeta+C\int_{\Tt^d}|D_xm^{\frac{\be_{n+1}+1}{2}}|^2
+\zeta\|m\|_{L^{q_n}(\Tt^d)}^{\beta_{n+1}+1}.
 \]
From \eqref{igm-eq6ae} and \eqref{igm-eq3ae}, taking $\delta$ and $\zeta$ small enough, it follows that 
for some $\delta_1>0$
\begin{align*}
&\int_{\Tt^d} m^{\be_{n+1}+1}(x,\tau)dx +\delta_1 \int_0^\tau \|m\|_{L^{q_n}(\Tt^d)}^{\be_{n+1}+1}dt\\
\leq  &C+C\int_{\Tt^d} m^{\be_{n+1}+1}(x,0)dx \\&+ C\int_0^\tau\int_{\Tt^d} H m 
+C\int_0^\tau\int_{\Tt^d} \left(\tr D^2_{pp}HD^2uD^2u \right)m.
\end{align*}

Because the last two terms on the right-hand side are bounded by Propositions \ref{pehm} and \ref{elsoe1}, we have the result.

Consider now the case $d=2$. Let $1<p<1+\dfrac 1\mu$. 
As before, we define inductively $\be_n$, starting with $\be_0=0$. 
Letting $\be_{n+1}:=\dfrac{p-1}p(\be_n+1)$,  we have that $\be_n$ is the $n^{th}$
partial sum of the geometric series with term $\dfrac{(p-1)^n}{p^n}$
and so $\lim\limits_{n\to \infty}\be_n=p-1$. Let $q_n=\frac{p(\be_{n+1}+1)}{1+\mu-\mu p}.$
For $\lam_n$ as in \eqref{eq:exponente} we have $$
\|m \|_{L^{(2+\mu)\be_{n+1}+1}(\Tt^d)}\leq \|m\|_{L^{1+\be_n}(\Tt^d)}^{1-\lambda_n}\|m\|_{L^{q_n}(\Tt^d)}^{\lambda_n}.
$$ From the previous definitions it follows that $\lam_n ((2+\mu)\be_{n+1}+1)=1+\beta_{n+1}.$
Since $\| m \|_{1+\be_n}\leq C,$ we get
\begin{equation}\label{eq:1}
\int_{\Tt^d} m^{(2+\mu)\be_{n+1} +1} dx
=\|m\|_{L^{(2+\mu)\be_{n+1}+1}(\Tt^d)}^{(2+\mu)\be_{n+1}+1}\leq C\|m\|_{L^{q_n}(\Tt^d)}^{\lam_n ((2+\mu)\be_{n+1}+1)}
=C\|m\|_{L^{q_n}(\Tt^d)}^{1+\beta_{n+1}}.
\end{equation}

As in  \eqref{igm-eq6ae}, using
\eqref{key-supere}, and \eqref{eq:1}
we get  for any $\tau\in[0,T]$
\begin{align}
\label{igm-eq6x} 
&\int_{\Tt^d} m^{\be_{n+1}+1}(x,\tau)dx +\frac{4\be_{n+1}}{\be_{n+1}+1}
\int_0^\tau\int_{\Tt^d}|D_xm^{\frac{\be_{n+1}+1}{2}}(x,t)|^2dx\,dt
\notag \\
&\leq  \int_{\Tt^d} m^{\be_{n+1}+1}(x,0)dx + C_\de \int_0^\tau\int_{\Tt^d} H m\notag\\ 
&+C_\de \int_0^\tau\int_{\Tt^d} \left(\tr D^2_{pp}HD^2uD^2u \right)m
+\delta \int_0^\tau \|m\|_{L^{q_n}(\Tt^d)}^{1+\beta_{n+1}} dt.
\end{align}

By Sobolev Theorem we get
\begin{align}\label{eq:doubt}
\|m\|_{L^{q_n}(\Tt^d)}^{1+\beta_{n+1}}=\|m^{\frac{\be_{n+1}+1}{2}}\|_{L^{\frac{2q_n}{\be_{n+1}+1}}(\Tt^d)}^2
&\le C\int_{\Tt^d} |D_xm^{\frac{\be_{n+1}+1}{2}}|^2 dx+C_\zeta+\zeta \|m\|_{L^{q_n}(\Tt^d)}^{\be_{n+1}+1}.
\end{align}

From \eqref{igm-eq6x} and \eqref{eq:doubt}, taking $\delta$ and $\zeta$ small enough
we have for some $\delta_1>0$
\begin{align*}
&\int_{\Tt^d} m^{\be_{n+1}+1}(x,\tau)dx +\delta_1\int_0^\tau \|m\|_{L^{q_n}(\Tt^d)}^{\be_{n+1}+1}dt\\
&\leq  C+C\int_{\Tt^d} m^{\be_{n+1}+1}(x,0)dx + C \int_0^\tau\int_{\Tt^d} H m\\ &
+C\int_0^\tau\int_{\Tt^d} \left(\tr D^2_{pp}HD^2uD^2u \right)m.
\end{align*}
Notice that the last two terms in the right-hand side are bounded, because of Propositions \ref{pehm} and \ref{elsoe1}. Then, we have established the result. 
\end{proof}

\subsection{Regularity by $L^p$ estimates}
\label{elep}

Now we bound $m^\epsilon$ in $L^\infty([0,T],L^p(\Tt^d))$ with estimates depending polynomially on the $L^\infty$-norm of $D_pH$. 
Because explicit expressions will be needed, we prove them in detail. For ease of presentation, we omit the $\epsilon$ in the proofs of this Section.

We start by setting $1\;\leq \;\be_{0}\;<\;\frac{d(1+\mu)}{d(1+\mu)-2}$, and consider $\be_1\doteq \theta\be_0$, for some $\theta>1$ fixed.

\begin{Lemma}\label{lpi1}
Assume that $(u^\epsilon,m^\epsilon)$ is a solution of \eqref{eq:smfg}
and let $\be\ge\be_0$ for $\be_0>1$ fixed. Then 
\begin{align}\label{eq:eq02}
\frac{d}{dt}\int_{\Tt^d}\left(m^{\epsilon}\right)^\be(t,x)dx
&\leq C\left\|\left|D_{p}H\right|^{2}\right\|_{L^\infty(\Tt^d)}\int_{\Tt^d}\left(m^{\epsilon}\right)^{\be}(t,x)dx
-c\int_{\Tt^d}\left|D_{x}\left((m^{\epsilon})^{\frac{\be}{2}}\right)\right|^{2}dx.
\end{align}
\end{Lemma}

\begin{Lemma}\label{lpi2}
We have
$$\int_{\Tt^d}(m^\epsilon)^{\be_1}(\tau,x)dx\leq\left(\int_{\Tt^d}\left(m^\epsilon\right)^{\be_0}\left(\tau,x\right)dx\right)^{\theta\kappa}
\left(\int_{\Tt^d}\left(m^\epsilon\right)^{\frac{2^{*}\be_1}{2}}\left(\tau,x\right)dx\right)^{\frac{2\left(1-\kappa\right)}{2^{*}}},$$
where $\kappa$ is given by 
\begin{equation}\label{eq:eqkappa}
\kappa=\frac{2}{d\left(\theta-1\right)+2}.
\end{equation}
\end{Lemma}
\begin{proof}
H\"older's inequality gives
$$\left(\int_{\Tt^d}m^{\be_1}\right)^{\frac{1}{\be_1}}
\leq\left(\int_{\Tt^d}m^{\be_0}\right)^{\frac{\kappa}{\be_0}}\left(\int_{\Tt^d}m^{\frac{2^{*}}{2}\be_1}\right)^{\frac{\left(1-\kappa\right)}{\frac{2^{*}}{2}\be_1}},$$
where 
$\frac{1}{\theta\be_0}\;=\;\frac{\kappa}{\be_0}\;+\;\frac{2\left(1-\kappa\right)}{2^{*}\theta\be_0}$.
By rearranging the exponents the inequality in the statement follows. The expression for $\kappa$ follows from the previous identity.
\end{proof}

\begin{Lemma}\label{lpi3}
Let $\kappa$ be defined by \eqref{eq:eqkappa}. Then 
$$\left(\int_{\Tt^d}\left(m^\epsilon\right)^{\be_1}\right)^{\left(1-\kappa\right)}
\leq C+\delta\left\|\left(m^\epsilon\right)^{\frac{\be_1}{2}}\right\|_{L^{2^{*}}(\Tt^d)}^{2\left(1-\kappa\right)}.$$
\end{Lemma}
\begin{proof}
Let 
$\lam\;=\;\frac{2}{d\left(\be_1-1\right)+2}.$
Then, 
$$\int_{\Tt^d}\left|m^{\be_1}\right|dx\leq
\left(\int_{\Tt^d}mdx\right)^{\be_1\lam}\left(\int_{\Tt^d}m^{\frac{2^{*}\be_1}{2}}dx\right)^{\frac{2\left(1-\lam\right)}{2^{*}}}.$$
Because $m$ is a probability measure for every $t\in\left[0,T\right]$
one obtains
$$\left(\int_{\Tt^d}m^{\be_1}dx\right)^{\left(1-\kappa\right)}\leq \left\|m^{\frac{\be_1}{2}}\right\|_{L^{2^{*}}(\Tt^d)}^{2\left(1-\kappa\right)\left(1-\lam\right)}.$$
Finally, since $\left(1-\lam\right)<1$ a further application of the Young's inequality weighted by $\delta$ establishes the result.
\end{proof}

\begin{Proposition}\label{ppi1}
We have
\begin{align*}
\int_{\Tt^d}\left(m^\epsilon\right)^{\be_1}dx\leq\left[\int_{\Tt^d}\left(m^\epsilon\right)^{\be_0}dx\right]^{\theta\kappa}
\left[C+C\left(\int_{\Tt^d}\left|D_{x}\left(\left(m^\epsilon\right)^{\frac{\be_1}{2}}\right)\right|^{2}dx\right)^{\left(1-\kappa\right)}\right],
\end{align*}where $\kappa$ is given by (\ref{eq:eqkappa}).
\end{Proposition}
\begin{proof}
Sobolev's Theorem implies that 
\begin{align}\label{eq:eq6}
\left\|\left(m\right)^{\frac{\be_1}{2}}\right\|_{L^{2^{*}}(\Tt^d)}^{2\left(1-\kappa\right)}
\leq C\left(\int_{\Tt^d}\left|D_{x}\left(\left(m\right)^{\frac{\be_1}{2}}\right)\right|^{2}\right)^{\left(1-\kappa\right)}
+C\left(\int_{\Tt^d}\left|\left(m\right)^{\be_1}\right|\right)^{\left(1-\kappa\right)}.
\end{align}
Using
Lemma \ref{lpi3} we obtain 
\begin{align}\notag
\left(\int_{\Tt^d}m^{\frac{2^{*}}{2}\be_1}\left(\tau,x\right)dx\right)^{\frac{2}{2^{*}}\left(1-\kappa\right)}
&=\left\|m^{\frac{\be_1}{2}}\right\|_{L^{2^{*}}(\Tt^d)}^{2\left(1-\kappa\right)}\\\label{eq:eq07}
&\le C\left(\int_{\Tt^d}\left|D_{x}\left(m^{\frac{\be_1}{2}}\right)\right|^{2}\right)^{\left(1-\kappa\right)}+ C.
\end{align}
By combining inequality \eqref{eq:eq07} with Lemma \ref{lpi2} the result follows.
\end{proof}

Next, we control the derivative with respect to time of $\left\|m^\epsilon\right\|_{L^{\be_1}(\Tt^d)}^{\be_1}$.

\begin{Proposition}\label{ppi2}
Let $(u^\epsilon,m^\epsilon)$ be a solution of \eqref{eq:smfg}. If $\kappa$ is given as in \eqref{eq:eqkappa}, then 
\begin{align}\label{eq:eq08}
\frac{d}{dt}\int_{\Tt^d}(m^\ep)^{\be_1}dx\leq C+C\left\|\left|D_{p}H\right|^{2}\right\|_{L^\infty(\Tt^d)}^{r}\left(\int_{\Tt^d}(m^\ep)^{\be_0}dx\right)^{\theta},
\end{align}where $r=\frac{1}{\kappa}.$
\end{Proposition}
\begin{proof}
Using
Lemma \ref{lpi1} with $\be\equiv\be_1$, and 
applying Proposition \ref{ppi1} we have
\begin{align*}
\frac{d}{dt}\int_{\Tt^d}m^{\be_1}
&\leq C\left\|\left|D_{p}H\right|^{2}\right\|_{L^\infty(\Tt^d)}\left(\int_{\Tt^d}m^{\be_0}\right)^{\theta\kappa}
\left[C\left(\int_{\Tt^d}\left|D_{x}\left(m^{\frac{\be_1}{2}}\right)\right|^{2}dx\right)^{\left(1-\kappa\right)}+C\right]\\
&-c\int_{\Tt^d}\left|D_{x}\left(m^{\frac{\be_1}{2}}\right)\right|^{2}\\
&\leq C\left\|\left|D_{p}H\right|^{2}\right\|_{L^\infty(\Tt^d)}\left(\int_{\Tt^d}m^{\be_0}\right)^{\theta\kappa}
\left(\int_{\Tt^d}\left|D_{x}\left(m^{\frac{\be_1}{2}}\right)\right|^{2}\right)^{\left(1-\kappa\right)}\\
&+ C\left\|\left|D_{p}H\right|^{2}\right\|_{L^\infty(\Tt^d)}\left(\int_{\Tt^d}m^{\be_0}\right)^{\theta\kappa}
-c\int_{\Tt^d}\left|D_{x}\left(m^{\frac{\be_1}{2}}\right)\right|^{2}\\
&\leq C\left\|\left|D_{p}H\right|^{2}\right\|_{L^\infty(\Tt^d)}\left(\int_{\Tt^d}m^{\be_0}\right)^{\theta\kappa}
+C\left\|\left|D_{p}H\right|^{2}\right\|_{L^\infty(\Tt^d)}^{r}\left(\int_{\Tt^d}m^{\be_0}\right)^{\theta}\\
&\leq C+C\left\|\left|D_{p}H\right|^{2}\right\|_{L^\infty(\Tt^d)}^{r}\left(\int_{\Tt^d}m^{\be_0}\right)^{\theta},
\end{align*}
where the last two inequalities follow by applying Young's inequality 
with $\varepsilon$ for the conjugate exponents $r$ and $s$ given by
$s=\frac{1}{1-\kappa}$ and $r=\frac{1}{\kappa}.$
\end{proof}

\begin{Corollary}\label{cpi1}
Suppose that $(u^\epsilon,m^\epsilon)$ is a solution of
\eqref{eq:smfg}.
Let $r$ be given as in
Proposition \ref{ppi2}. Then 
\begin{align*}
\int_{\Tt^d}m^{\be_1}(\tau,x)dx\leq C+C\left\|\left|D_{p}H\right|^{2}\right\|_{L^\infty(\Tt^d)}^{r}.
\end{align*}
\end{Corollary}
\begin{proof}
Integrate (\ref{eq:eq08}) in time over $\left(\tau,T\right)$. This yields
\begin{align}\label{eq:cor1}
\int_{\Tt^d}m^{\be_1}\left(\tau,x\right)dx
&\leq C\left\|\left|D_{p}H\right|^{2}\right\|_{L^\infty(\Tt^d)}^{r}\int_{0}^{\tau}\left(\int_{\Tt^d}m^{\be_0}dx\right)^{\theta}dt+C.
\end{align}
From Proposition \ref{emsupere}, we have $\int_{\Tt^d}m^{\be_{0}}\left(\tau,x\right)dx\leq C$. The result is then established.
\end{proof}

\subsection{Interpolated bounds}\label{sec84}

We now obtain estimates for $m^\ep$ in terms of the $L^\infty$-norm of $Du^\ep$ by interpolating previous results. 

\begin{Lemma}\label{lem:lem1sec84}
Let $(u^\epsilon,m^\epsilon)$ be a solution of \eqref{eq:smfg}. Assume A\ref{ah}-\ref{bcc} and A\ref{aspdpph} hold. 
Assume further that $\theta,\,p,\,r>1$, $0\leq \upsilon\leq 1$ are such that \eqref{eq:eq3sec84}-\eqref{eq:eq4sec84}.
Let $\beta_\upsilon=\frac{\theta\beta_0}{\theta+\upsilon-\theta\upsilon},$ where $\beta_0\in\left[1,\frac{d(1+\mu)}{d(1+\mu)-2}\right)$.
Then
\[
\|g(m^\ep)\|_{L^\infty(0,T;L^p(\Tt^d))}\leq C+C\left\|\left|D_pH\right|^2\right\|_{L^\infty(0,T;L^\infty(\Tt^d))}^{\frac{r\upsilon\af}{\theta\be_0}}.
\]
\end{Lemma}
\begin{proof}
As before, we omit the $\ep$ in the proof. H\"older's inequality gives
\[
\left(\int_{\Tt^d}m^{\beta_\upsilon}\right)^\frac{1}{\beta_\upsilon}\leq \left(\int_{\Tt^d}m^{\beta_0}\right)^\frac{1-\upsilon}{\beta_0}\left(\int_{\Tt^d}m^{\theta\beta_0}\right)^\frac{\upsilon}{\theta\beta_0},
\]since $\frac{1}{\beta_\upsilon}=\frac{1-\upsilon}{\be_0}+\frac{\upsilon}{\theta\be_0}$.

Theorem \ref{emsupere} ensures that 
\[
\int_{\Tt^d}m^{\beta_\upsilon}\leq C\left(\int_{\Tt^d}m^{\theta\beta_0}\right)^\frac{\upsilon}{\theta+\upsilon-\theta\upsilon}.
\]
On the other hand, Corollary \ref{cpi1} gives
\[
\int_{\Tt^d}m^{\theta\beta_0}\leq C+C\left\|\left|D_pH\right|^2\right\|_{L^\infty(0,T;L^\infty(\Tt^d))}^{r},
\]
which in turn leads to 
\[
\int_{\Tt^d}m^{\beta_\upsilon}\leq C+C\left\|\left|D_pH\right|^2\right\|_{L^\infty(0,T;L^\infty(\Tt^d))}^\frac{r\upsilon}{(\theta+\upsilon-\theta\upsilon)}.
\]
Note that $\left\|g(m)\right\|_{L^p(\Tt^d)}=\left(\int_{\Tt^d}m^{\alpha p}\right)^\frac{1}{p}$. Because of \eqref{eq:eq4sec84} it follows that 
\[\left\|g(m)\right\|_{L^p(\Tt^d)}\leq C+C\left\|\left|D_pH\right|^2\right\|_{L^\infty(0,T;L^\infty(\Tt^d))}^\frac{r\upsilon}{p(\theta+\upsilon-\theta\upsilon)}.\]By noticing that $p(\theta+\upsilon-\theta\upsilon)=\frac{\theta\be_0}{\af},$ because of \eqref{eq:eq4sec84}, the Lemma is established.
\end{proof}
\begin{proof}[Proof of Theorem \ref{cor:cor2sec84}.]
Theorem \ref{cor:cor2sec84} follows from Lemma \ref{lem:lem1sec84} and A\ref{adph}.
\end{proof}

\section{Bounds for the Hamilton-Jacobi equation}
\label{pub}


Next we control $\left\|u^\epsilon\right\|_{L^\infty(\Tt^d\times\left[0,T\right])}$ in terms of $\left\|g_\epsilon(m^\ep)\right\|_{L^\infty(\left[0,T\right],L^p(\Tt^d))}$. Because we already have lower bounds for $u^\ep$,  since $g\geq 0$, see \cite{GPM2}, it suffices in what
follows to obtain upper bounds. 

We start by presenting the proof of Lemma \ref{imphc0}.
\begin{proof}[Proof of Lemma \ref{imphc0}.]
 For ease of notation, we omit the $\epsilon$ in $m^\ep$. By using Proposition \ref{plb} with $b=0$ and $\zeta_0=\te(\cdot,\tau)=\delta_{x}$, 
$0\leq \tau<T$ we obtain the estimate
\begin{align*}
u(x,\tau)\le &(T-\tau)\max_{z\in \Tt^d} L(z,0)\\
&+
\int_\tau^T\int_{\Tt^d}g_\ep(m)(y,t)\te(y,t-\tau)dydt
+\int_{\Tt^d} u(y,T)\te(y,T-\tau)dy.
\end{align*}
The main issue is to control
\begin{equation}
\label{AAA}
\int_\tau^T\int_{\Tt^d} g_\ep(m)(y,t)\te(y,t-\tau)dydt. 
\end{equation}
For $\dfrac 1 p+\dfrac 1 q=1$, the heat kernel satisfies
\[
\|\te(\cdot, t)\|_q\leq \frac{C}{t^{\frac{d}{2p}}}.
\]
Hence,
\[
\int_{\Tt^d} g_\ep(m)(y,t)\te(y,t-\tau)dy
\leq \frac{C}{(t-\tau)^{\frac{d}{2p}}}\|g_\ep(m(\cdot,t))\|_{L^p(\Tt^d)}. 
\]
Thus if $d<2p$ we have
\[
\int_\tau^T\int_{\Tt^d}g_\ep(m)(y,t)\te(y,t-\tau)dydt\le 
C\|g_\ep(m)\|_{L^\infty(0,T;L^{p}(\Tt^d))}.
\]
\end{proof}

\section{Regularity by the adjoint method}
\label{ram}

The aim of this Section is to obtain estimates for $\|Du^\ep\|_{L^\infty(0,T;L^\infty(\Tt^d))}$. The key tools are the 
adjoint method \cite{E3}, and the methods developed in \cite{GM} (see also \cite{GPatVrt}). In what follows we obtain Lipschitz estimates for the solutions of the Hamilton-Jacobi equation in terms of $L^\infty(0,T;L^p(\Tt^d))$ norms of the nonlinearity $g$. This result is important not only for its role in the realm of the mean-field games theory, but also adds to the current understanding of the regularity of superquadratic Hamilton-Jacobi equations. For some related results, see \cite{Bsuper}, \cite{CDLP}, and \cite{CarSil}, where the authors investigate H\"older regularity.

Our main a-priori estimate is the following:
\begin{teo}
\label{uReg1}
Suppose  A\ref{ah}-A\ref{aspdpph} hold.
Let $(u^\ep,m^\ep)$ be a solution of \eqref{eq:smfg} and assume that $p>d$.  
Then 
\begin{align*}
\|Du^\ep\|_{L^\infty(0,T;L^\infty(\Tt^d))}\leq C+& C\|u^\ep\|_{L^\infty(0,T;L^\infty(\Tt^d))}\\
+C\|g_\ep(m^\ep)\|_{L^\infty(0,T;L^p(\Tt^d))}&\Big(\|Du^\ep\|^\mu_{L^{\infty}(0,T;L^{\infty}(\Tt^d))}(1+\|u^\ep\|_{L^{\infty}(0,T;L^{\infty}(\Tt^d))})\Big),
\end{align*}
where $\mu$ is the exponent given by Assumption A\ref{adph}. 
\end{teo}
\begin{proof}
For convenience, the proof of the Theorem proceeds in the four steps below.
\end{proof}
We omit the superscript $\ep$ for the solution  $(u^\ep,m^\ep)$ in the following proofs.

\paragraph{Step 1}

The adjoint equation is the following partial differential equation
\begin{equation}
\label{ADJ2}
\rho_t-\Delta \rho -\div(D_pH \rho)=0
\end{equation}
for which we choose the initial data
$\rho(\cdot,\tau)=\delta_{x_0}$.
Using this and the first equation in \eqref{eq:smfg}, we have the following representation formula for $u$:  
\begin{equation}
\label{7A2}
u(x_0,\tau)=\int_\tau^T\int_{\Tt^d}(D_pH D_xu-H+g_\ep(m))\rho+
\int_{\Tt^d}u(x,T) \rho(x,T). 
\end{equation}
\begin{Corollary}
\label{cram11}
Suppose  A\ref{ah}-A\ref{aspdpph} hold.
Let $(u^\epsilon,m^\ep)$ be a solution of \eqref{eq:smfg}. Let $\rho$ solve
\eqref{ADJ2} with initial data $\rho(\cdot,\tau)=\delta_{x_0}$.  Then 
\begin{equation}
\label{7B1}
\int_\tau^T\int_{\Tt^d}H\rho+\int_\tau^T\int_{\Tt^d} g_\ep(m)\rho
\le C+C\left[u(x_0,\tau)-\int_{\Tt^d} u(x,T)\rho(x,T)\right]. 
\end{equation}
\end{Corollary}
\begin{proof}
It suffices to use Assumption A\ref{aele} in \eqref{7A2}.
\end{proof}

\begin{Corollary}
\label{cram12}
Suppose  A\ref{ah}-A\ref{aspdpph} hold.
Let $(u^\epsilon,m^\ep)$ be a solution of \eqref{eq:smfg}. Let $\rho$ solve
\eqref{ADJ2} with initial data $\rho(\cdot,\tau)=\delta_{x_0}$.  Then 
\begin{equation}
\label{7B2}
\int_\tau^T\int_{\Tt^d}H\rho\le C+C\|u^\ep\|_{L^{\infty}(0,T;L^{\infty}(\Tt^d))}
\end{equation}
\end{Corollary}
\begin{proof}
The result follows from Corollary \ref{cram11} and the positivity of $g$.
\end{proof}

\paragraph{Step 2}

We have, using the ideas from \cite{GM}:
\begin{Proposition}
\label{pram1}
Suppose  A\ref{ah}-A\ref{aspdpph} hold.
Let $(u^\ep,m^\ep)$ be a solution of \eqref{eq:smfg}. 
Let $\rho$ solve \eqref{ADJ2} with initial data $\rho(\cdot,\tau)=\delta_{x_0}$. 
Then, 
for $0<\nu<1$
\[
\int_\tau^T\int_{\Tt^d}|D\rho^{\nu/2}|^2dx\,dt\le C+C\|Du^\ep\|^{\mu}_{L^{\infty}(0,T;L^{\infty}(\Tt^d))}\Big(1+\|u^\ep\|_{L^{\infty}(0,T;L^{\infty}(\Tt^d))}\Big),
\]where $\mu$ is the exponent given in Assumption A\ref{adph}.
\end{Proposition}
\begin{proof}
Multiply \eqref{ADJ2} by $\nu\rho^{\nu-1}$. Then
\begin{equation}
\label{blabla1}
\frac{\partial\rho^\nu}{\partial t}-\nu\rho^{\nu-1}\div(D_pH(x,Du)\rho)=
\nu\rho^{\nu-1}\Delta \rho. 
\end{equation}
We now integrate the previous identity on $[\tau,T]\times \Tt^d$.
Since $\rho(\cdot, t)$ is a probability measure and we have $0<\nu<1$, it follows that:
$\int_{\Tt^d} \rho^\nu(x,t) dx\leq 1$.
Consequently, the integral of the first term of the left hand side of \eqref{blabla1}
is bounded. We also have:
\begin{align*}
&\left|\int_\tau^T \int_{\Tt^d}\nu\rho^{\nu-1}\div(D_pH(x,Du)\rho) dx dt\right|\\
&=c_{\nu}\left|\int_\tau^T\int_{\Tt^d} \rho^{\nu/2} \rho^{\nu/2-1}D\rho D_p H dx dt\right|\\ 
&\le\zeta\int_\tau^T\int_{\Tt^d}|D(\rho^{\nu/2})|^2dxdt+
C_{\zeta,\nu}\int_\tau^T\int_{\Tt^d} |D_pH|^2 \rho^{\nu}dx dt,  
\end{align*}
for any $\zeta>0$, with $C_{\zeta,\nu}$ depending only on $\zeta$ and $\nu$. 
Because $0<\nu<1$, we have $\rho^\nu \leq C_\delta+\delta \rho,$ for any $\delta>0$ and suitable $C_\delta$. 
Using Assumption A\ref{adph}, it follows from Proposition \ref{pehm} and Corollary \ref{cram12} that
\begin{align*}
 C\int_\tau^T\int_{\Tt^d} |D_pH|^2\rho^{\nu} dx\,dt&\le C
+C_\delta \int_\tau^T\int_{\Tt^d} |Du|^\mu H dx dt+\delta\int_\tau^T\int_{\Tt^d} |Du|^\mu H \rho dx dt\\
&\le C+C\|Du\|^{\mu}_{L^{\infty}(0,T;L^{\infty}(\Tt^d))}\Big(1+\|u\|_{L^{\infty}(0,T;L^{\infty}(\Tt^d))}\Big).
\end{align*}

The integral of the right hand side of \eqref{blabla1} is
\[
\nu (1-\nu)\int_\tau^T\int_{\Tt^d}|D\rho|^2\rho^{\nu-2} dx dt=
\frac{4(1-\nu)}{\nu}\int_\tau^T\int_{\Tt^d}|D(\rho^{\nu/2})|^2 dx dt. 
\] 

Gathering the previous estimates we get 
\begin{align*}
\frac{4 (1-\nu)}{\nu} &\int_\tau^T\int_{\Tt^d}|D (\rho^{\nu/2})|^2 dx dt
\le C+\zeta\int_\tau^T \int_{\Tt^d}|D(\rho^{\nu/2})|^2 dx dt\\
&+C\|Du\|^{\mu}_{L^{\infty}(0,T;L^{\infty}(\Tt^d))}(1+\|u\|_{L^{\infty}(0,T;L^{\infty}(\Tt^d))}).
\end{align*}
Choosing $\zeta$ small enough we obtain the result. 
\end{proof}

\paragraph{Step 3}

To finish the proof of Theorem \ref{uReg1} fix now a unit vector $\xi\in\Rr^d$.
Differentiate the first equation of \eqref{eq:smfg} in the $\xi$
direction and multiply it by $\rho$.
Integrating by parts and using \eqref{ADJ2} we obtain:
 \[
u_\xi(x_0,\tau)=\int_\tau^T\int_{\Tt^d}-D_\xi H\rho+(g_\ep(m))_\xi\rho
+\int_{\Tt^d}u_\xi(x,T)\rho(x,T). 
\] 
Note that
\[
\left|\int_{\Tt^d} u_\xi(x,T) \rho(x,T)\right|\le \|u_\xi(\cdot,T)\|_{L^{\infty}(\Tt^d)}.
\]
Using Corollary \ref{cram12} and  Assumption A\ref{strong} we have
\[
\int_\tau^T\int_{\Tt^d}|D_\xi H|\rho\leq C+C\int_\tau^T\int_{\Tt^d}H\rho\leq C+C\|u\|_{L^\infty(0,T;L^{\infty}(\Tt^d))}. 
\]
Thus it remains to bound
\begin{equation}
\label{rme0}
\int_\tau^T\int_{\Tt^d} (g_\ep(m))_\xi\rho.
\end{equation}
This will be done in the next step. 

\paragraph{Step 4}

To bound \eqref{rme0} we integrate by parts, from which 
it follows that:
\begin{align*}
\left|\int_\tau^T\int_{\Tt^d}(g_\ep(m))_\xi\rho\right|\leq 
&\int_\tau^T\int_{\Tt^d} g_\ep(m)\rho^{1-\be} |\rho^{\be-1}D\rho|\\
&\leq C\int_\tau^T \|g_\ep(m)\|_a\|\rho^{1-\be}\|_b\|D\rho^\be\|_2,
\end{align*}
for any $2\leq a, b\leq \infty$ satisfying
$\frac 1 a +\frac 1 b+ \frac 1 2 =1$. 
From this we get, for $\be=\frac \nu 2$, with $0<\nu<1$, 
\begin{align*}
&\left|\int_\tau^T \int_{\Tt^d} g_\ep(m)_\xi \rho\right| \leq 
C\|g(m)\|_{L^\infty(\tau,T;L^{a}(\Tt^d))}
\|\rho^{1-\frac \nu 2}\|_{L^2(\tau,T;L^b(\Tt^d))}
\|D\rho^{\frac \nu 2}\|_{L^2(\tau,T;L^2(\Tt^d))}.
\end{align*}
From Proposition \ref{pram1} we have a bound for
$\|D\rho^{\frac\nu 2}\|_{L^2(\tau,T;L^2(\Tt^d))}$. Therefore, it suffices to estimate $\|\rho^{1-\frac \nu 2}\|_{L^2(\tau,T;L^b(\Tt^d))}.$ We have now to estimate
\[
\int_\tau^T\Big(\int_{\Tt^d} \rho^{b( 1-\frac \nu 2)}\Big)^\frac{2}{b}.
\]
Given $0<\kappa<1$, we define $b$ by
\begin{equation}\label{eq:k1}
\frac{1}{b( 1-\frac \nu 2)}=1-\kappa+\frac{\kappa}{\frac{2^*\nu}{2}}. 
\end{equation}
We will choose $\kappa$ appropriately so that $b>2$ holds. 
Additionally, it follows trivially from \eqref{eq:k1} that 
$1<b (1-\frac \nu 2) < \frac{2^*}2 \nu$,  
and so by H\"older's inequality 
we have:
\[
\Big(\int_{\Tt^d} \rho^{b( 1-\frac \nu 2)}\Big)^\frac{1}{b( 1-\frac \nu 2)}
\leq 
\Big(\int_{\Tt^d} \rho\Big)^{1-\kappa}
\Big(\int_{\Tt^d} \rho^{\frac{2^*\nu}{2}}\Big)^{\frac{2\kappa}{2^* \nu}}. 
\]
Recall that by Sobolev's inequality,
$\Big(\int_{\Tt^d} \rho^{\frac{2^*\nu}{2}}\Big)^{\frac{2}{2^*}}\leq C+C \int_{\Tt^d} |D\rho^{\frac \nu 2}|^2$.
Choose now $\kappa=\frac{\nu}{2-\nu}.$
Note that if $0<\nu<1$ we have $0<\kappa<1$. 
Then 
\begin{align*}
&\Big\|\rho^{1-\frac{\nu}{2}}\Big\|_{L^2(0,T;L^b(\Tt^d))}\leq   \left[C+C \int_0^T\int_{\Tt^d} |D\rho^{\frac \nu 2}|^2\right]^{\frac{1}{2}} \\
&\leq \left[C+C\|Du\|^\mu_{L^{\infty}(0,T;L^{\infty}(\Tt^d))}\Big(1+\|u\|_{L^{\infty}(0,T;L^{\infty}(\Tt^d))}\Big)\right]^{\frac{1}{2}}.
\end{align*}
Also, using Proposition \ref{pram1} we have
\[
\|D\rho^{\frac{\nu}{2}}\|_{L^2(0,T;L^2(\Tt^d))}\leq
\left[C+C\|Du\|^\mu_{L^{\infty}(0,T;L^{\infty}(\Tt^d))}\Big(1+\|u\|_{L^{\infty}(0,T;L^{\infty}(\Tt^d))}\Big)\right]^{\frac{1}{2}}.
\]

It remains to check that it is possible to choose $\nu$ such that $b>2$. Indeed, for $\frac{d-1}d<\nu<1$ we have $\frac{d-1}{d+1}<\kappa<1$, and  $b=\frac{2d}{3d-2d\nu-2}>2.$ Note that $a$ is given by $a=\frac d{d(\nu-1)+1}.$
Thus if $p>d$ we have, for $\nu$ close enough to $1$ that $p>a$ and, therefore, 
this ends the proof of Theorem \ref{uReg1}.

The result in Theorem \ref{uReg1} can be further simplified, as stated in Theorem \ref{uReg1c}. We now present its proof.

\begin{proof}[Proof of Theorem \ref{uReg1c}.]
By recurring to Lemma \ref{imphc0}, Theorem \ref{uReg1} becomes
\begin{align*}
\|Du\|_{L^\infty(\Tt^d\times[0,T)}&\leq C+C\|g_\ep\|_{L^\infty(0,T;L^p(\Tt^d))}\\&\quad+
C\|g_\ep\|_{L^\infty(0,T;L^p(\Tt^d))}\|Du\|^\mu_{L^\infty(\Tt^d\times[0,T])}\\
&\quad+C\|g_\ep\|_{L^\infty(0,T;L^p(\Tt^d))}\|Du\|^{\mu}_{L^\infty(\Tt^d\times[0,T])}\|u^\ep\|_{L^\infty(\Tt^d\times[0,T])}.
\end{align*}Young's inequality yields then
\begin{align*}
\|Du\|_{L^\infty(\Tt^d\times[0,T])}&\leq C+C\|g_\ep\|_{L^\infty(0,T;L^p(\Tt^d))}+C\|g_\ep\|_{L^\infty(0,T;L^p(\Tt^d))}^\frac{1}{1-\mu}\\&\quad+C\|g_\ep\|_{L^\infty(0,T;L^p(\Tt^d))}^\frac{1}{1-\mu}\|u^\ep\|_{L^\infty(\Tt^d\times[0,T])}^\frac{1}{1-\mu}.
\end{align*}
A further application of Young's inequality implies the result.
\end{proof}

\section{Lipschitz regularity for the Hamilton-Jacobi equation}
\label{sec:sec12}

In what follows we combine the results of Section \ref{igain} with the arguments from \ref{ram} to obtain Lipschitz regularity for the Hamilton-Jacobi equation.
\begin{Lemma}\label{lem:lem1sec12}
Let $(u^\epsilon,m^\epsilon)$ be a solution of \eqref{eq:smfg}-\eqref{itvp}. Suppose that
A\ref{ah}-\ref{alpha3} hold. 
Let $\theta,\tilde{\te},>1$, $0\leq \upsilon,
\tilde{\upsilon},\leq 1$. Let $r=r_\te$, $\tilde{r}=r_{\tilde{\te}}$
be given by \eqref{eq:eq3sec84} and $p_{\upsilon,\te}$,
$p_{\tilde{\upsilon},\tilde{\te}}$ be given by 
\eqref{eq:eq4sec84}. Suppose that $p_{\upsilon,\te}>d$, $p_{\tilde{\upsilon},\tilde{\te}}>\frac{d}{2}$. 
Then
\[
\left\|Du^\epsilon\right\|_{L^\infty(0,T;L^\infty(\Tt^d))}\leq C+C\left\|Du^\epsilon\right\|_{L^\infty(0,T;L^\infty(\Tt^d))}^{\frac{(2+2\mu)r\upsilon\af}{(1-\mu)\theta\be_0}+\frac{(2+2\mu)\tilde{r}\tilde{\upsilon}\af}{(1-\mu)\tilde{\theta}\be_0}}.
\]
\end{Lemma}
\begin{proof}
For ease of presentation, we remove the $\epsilon$. 
Theorem \ref{uReg1c} implies
\begin{align*}
\|Du\|_{L^\infty(0,T;L^\infty(\Tt^d))}\leq &
C+C\|g(m)\|^\frac{1}{1-\mu}_{L^\infty(0,T;L^p(\Tt^d))}\\
&+C\|g(m)\|^\frac{1}{1-\mu}_{L^\infty(0,T;L^p(\Tt^d))}\|u\|_{L^\infty(0,T;L^\infty(\Tt^d))}^\frac{1}{1-\mu}.
\end{align*}
Because $\tilde{p}>\frac{d}{2}$, we have from Lemma \ref{imphc0} that
$$\left\|u\right\|_{L^\infty(0,T;L^\infty(\Tt^d))}\leq
C+C\left\|g(m)\right\|_{L^\infty(0,T;L^{\tilde{p}}(\Tt^d))}.$$ 
By combining these, one obtains
\begin{align*}
\left\|Du\right\|_{L^\infty(0,T;L^\infty(\Tt^d))}&\leq
C+C\left\|g(m)\right\|^\frac{1}{1-\mu}_{L^\infty(0,T;L^p(\Tt^d))}\\
&\quad+C\left\|g(m)\right\|^\frac{1}{1-\mu}_{L^\infty(0,T;L^p(\Tt^d))}
\left\|g(m)\right\|^\frac{1}{1-\mu}_{L^\infty(0,T;L^{\tilde{p}}(\Tt^d))}.
\end{align*}
From Theorem \ref{cor:cor2sec84} it follows that 
\begin{align*}
\left\|Du\right\|_{L^\infty(0,T;L^\infty(\Tt^d))}&\leq C+C\left\|Du\right\|^{\frac{(2+2\mu)r\upsilon\af}{(1-\mu)\theta\be_0}+\frac{(2+2\mu)\tilde{r}\tilde{\upsilon}\af}{(1-\mu)\tilde{\theta}\be_0}}_{L^\infty(0,T;L^\infty(\Tt^d))},
\end{align*}which establishes the result.
\end{proof}

\begin{Proposition}\label{teo:teo1sec12}
Let $(u^\epsilon,m^\epsilon)$ be a solution of \eqref{eq:smfg}-\eqref{itvp}. Assume that
A\ref{ah}-\ref{alpha3} hold. 
Let $\theta,\tilde{\te}>1$, $0\leq \upsilon,
\tilde{\upsilon},\leq 1$. Let $r=r_\te$, $\tilde{r}=r_{\tilde{\te}}$ be given by 
\eqref{eq:eq3sec84} and $p_{\upsilon,\te}$, $p_{\tilde{\upsilon},\tilde{\te}}$
be given by \eqref{eq:eq4sec84}. Suppose $p_{\upsilon,\te}>d$ and
$p_{\tilde{\upsilon},\tilde{\te}}>\frac{d}{2},$ and there is
$n\in\mathbb{N}$ such that \eqref{eq:eq1sec12} is satisfied, then
$Du^\epsilon\in L^\infty(\Tt^d\times\left[0,T\right])$. 
\end{Proposition}
\begin{proof}
Lemma \ref{lem:lem1sec12} ensures that 
\[
\left\|Du^\ep\right\|_{L^\infty(0,T;L^\infty(\Tt^d))}\leq C+C\left\|Du^\ep\right\|_{L^\infty(0,T;L^\infty(\Tt^d))}^{\frac{(2+2\mu)r\upsilon\af}{(1-\mu)\theta\be_0}+\frac{(2+2\mu)\tilde{r}\tilde{\upsilon}\af}{(1-\mu)\tilde{\theta}\be_0}}.
\]Because of \eqref{eq:eq1sec12},  the result follows using Young's inequality.
\end{proof}

The results in this Section strongly rely on several constraints involving the various parameters of the problem. It is critical to ensure that this set of constraints can be mutually satisfied. This is done in the following Lemma:

\begin{Lemma}\label{techlemma}
If $$\af<\frac{2}{d(1+\mu)-2},$$
then there exist $1<\theta,\tilde{\theta}$ and
$0\leq \upsilon,\,\tilde{\upsilon}\leq 1$ such that 
for $r=r_\te$, $\tilde{r}=r_{\tilde{\te}}$ given by \eqref{eq:eq3sec84}
and $p=p_{\upsilon,\te}$, $\tilde{p}=p_{\tilde{\upsilon},\tilde{\te}}$
given by \eqref{eq:eq4sec84} 
we have that $p>d$,  $\tilde{p}>\frac d2$ and
\begin{equation}\label{eq:eq1sec12}
\frac{(2+2\mu)r\upsilon\af}{(1-\mu)\theta\be_0}+\frac{(2+2\mu)\tilde{r}\tilde{\upsilon}\af}{(1-\mu)\tilde{\theta}\be_0}<1,
\end{equation}
are satisfied. 
\end{Lemma}
\begin{proof}
To establish the Lemma we use the symbolic software Mathematica. See \cite{PIM} for details. 
\end{proof}

We can now end the paper with the proof of Theorem \ref{teo:sec12}:

\begin{proof}[Proof of Theorem \ref{teo:sec12}.]
It remains to check that \eqref{eq:eq3sec84} as well as \eqref{eq:eq4sec84} and \eqref{eq:eq1sec12} hold simultaneously. In fact, under A\ref{alpha3} it follows from Lemma \ref{techlemma}.
\end{proof}

\bibliographystyle{plain}
\bibliography{mfg}

\begin{thebibliography}{10}

\bibitem{achdou2013finite}
Y.~Achdou.
\newblock Finite difference methods for mean field games.
\newblock In {\em Hamilton-Jacobi Equations: Approximations, Numerical Analysis
  and Applications}, pages 1--47. Springer, 2013.

\bibitem{CDY}
Y.~Achdou, F.~Camilli, and I.~Capuzzo-Dolcetta.
\newblock Mean field games: numerical methods for the planning problem.
\newblock {\em SIAM J. Control Optim.}, 50(1):77--109, 2012.

\bibitem{DY}
Y.~Achdou and I.~Capuzzo-Dolcetta.
\newblock Mean field games: numerical methods.
\newblock {\em SIAM J. Numer. Anal.}, 48(3):1136--1162, 2010.

\bibitem{CrAm}
H.~Amann and M.~G. Crandall.
\newblock On some existence theorems for semi-linear elliptic equations.
\newblock {\em Indiana Univ. Math. J.}, 27(5):779--790, 1978.

\bibitem{Bsuper}
G.~Barles.
\newblock A short proof of the {$C^{0,\alpha}$}-regularity of viscosity
  subsolutions for superquadratic viscous {H}amilton-{J}acobi equations and
  applications.
\newblock {\em Nonlinear Anal.}, 73(1):31--47, 2010.

\bibitem{CDLP}
I.~Capuzzo~Dolcetta, F.~Leoni, and A.~Porretta.
\newblock H\"older estimates for degenerate elliptic equations with coercive
  {H}amiltonians.
\newblock {\em Trans. Amer. Math. Soc.}, 362(9):4511--4536, 2010.

\bibitem{cardaliaguet}
P.~Cardaliaguet.
\newblock Notes on mean-field games.
\newblock 2011.

\bibitem{Cd1}
P.~Cardaliaguet.
\newblock Long time average of first order mean field games and weak {KAM}
  theory.
\newblock {\em Dyn. Games Appl.}, 3(4):473--488, 2013.

\bibitem{Cd2}
P.~Cardaliaguet.
\newblock Weak solutions for first order mean-field games with local coupling.
\newblock {\em Preprint}, 2013.

\bibitem{CLLP}
P.~Cardaliaguet, J.-M. Lasry, P.-L. Lions, and A.~Porretta.
\newblock Long time average of mean field games.
\newblock {\em Netw. Heterog. Media}, 7(2):279--301, 2012.

\bibitem{CarSil}
P.~Cardaliaguet and L.~Silvestre.
\newblock H\"older continuity to {H}amilton-{J}acobi equations with
  superquadratic growth in the gradient and unbounded right-hand side.
\newblock {\em Comm. Partial Differential Equations}, 37(9):1668--1688, 2012.

\bibitem{Carmona2}
R.~Carmona and F.~Delarue.
\newblock Mean field forward-backward stochastic differential equations.
\newblock {\em Electron. Commun. Probab.}, 18:no. 68, 15, 2013.

\bibitem{Carmona1}
Ren{\'e} Carmona and Fran{\c{c}}ois Delarue.
\newblock Probabilistic analysis of mean-field games.
\newblock {\em SIAM J. Control Optim.}, 51(4):2705--2734, 2013.

\bibitem{E2}
L.~C. Evans.
\newblock Further {PDE} methods for weak {KAM} theory.
\newblock {\em Calc. Var. Partial Differential Equations}, 35(4):435--462,
  2009.

\bibitem{E3}
L.~C. Evans.
\newblock Adjoint and compensated compactness methods for {H}amilton-{J}acobi
  {PDE}.
\newblock {\em Arch. Ration. Mech. Anal.}, 197(3):1053--1088, 2010.

\bibitem{GF}
Rita Ferreira and Diogo~A. Gomes.
\newblock On the convergence of finite state mean-field games through
  {$\Gamma$}-convergence.
\newblock {\em J. Math. Anal. Appl.}, 418(1):211--230, 2014.

\bibitem{GIMY}
D.~Gomes, R.~Iturriaga, H.~S{\'a}nchez-Morgado, and Y.~Yu.
\newblock Mather measures selected by an approximation scheme.
\newblock {\em Proc. Amer. Math. Soc.}, 138(10):3591--3601, 2010.

\bibitem{GMit}
D.~Gomes and H.~Mitake.
\newblock Stationary mean-field games with congestion and quadratic
  hamiltonians.
\newblock {\em preprint}.

\bibitem{GMS}
D.~Gomes, J.~Mohr, and R.~R. Souza.
\newblock Discrete time, finite state space mean field games.
\newblock {\em Journal de Math\'ematiques Pures et Appliqu\'ees},
  93(2):308--328, 2010.

\bibitem{GMS2}
D.~Gomes, J.~Mohr, and R.~R. Souza.
\newblock Continuous time finite state mean-field games.
\newblock {\em Appl. Math. and Opt.}, 68(1):99--143, 2013.

\bibitem{GPat}
D.~Gomes and S.~Patrizi.
\newblock Obstacle mean-field game problem.
\newblock {\em Preprint}, 2013.

\bibitem{GPatVrt}
D.~Gomes, S.~Patrizi, and V.~Voskanyan.
\newblock On the existence of classical solutions for stationary extended mean
  field games.
\newblock {\em Nonlinear Anal.}, 99:49--79, 2014.

\bibitem{GPim1}
D.~Gomes and E.~Pimentel.
\newblock Local regularity for mean-field games in the whole space.
\newblock {\em preprint}.

\bibitem{GPim2}
D.~Gomes and E.~Pimentel.
\newblock Time dependent mean-field games with logarithmic nonlinearities.
\newblock {\em preprint}.

\bibitem{GPM2}
D.~Gomes, E.~Pimentel, and H~Sanchez-Morgado.
\newblock Time dependent mean-field games in the subquadratic case.
\newblock {\em To appear in Comm. Partial Differential Equations}, 2014.

\bibitem{GM}
D.~Gomes and H.~S{\'a}nchez~Morgado.
\newblock A stochastic {E}vans-{A}ronsson problem.
\newblock {\em Trans. Amer. Math. Soc.}, 366(2):903--929, 2014.

\bibitem{GVrt}
D.~Gomes and V.~Voskanyan.
\newblock Extended deterministic mean-field games.
\newblock {\em Preprint}, 2013.

\bibitem{GPM1}
D.~A. Gomes, G.~E. Pires, and H.~S{\'a}nchez-Morgado.
\newblock A-priori estimates for stationary mean-field games.
\newblock {\em Netw. Heterog. Media}, 7(2):303--314, 2012.

\bibitem{GS}
Diogo~A. Gomes and Jo{\~a}o Sa{\'u}de.
\newblock Mean {F}ield {G}ames {M}odels---{A} {B}rief {S}urvey.
\newblock {\em Dyn. Games Appl.}, 4(2):110--154, 2014.

\bibitem{GueantT}
O.~Gu{\'e}ant.
\newblock {\em Mean Field Games and Applications to Economics}.
\newblock Ph.D. Thesis. Universit\'e Paris Dauphine, Paris, 2009.

\bibitem{Ge}
O.~Gu{\'e}ant.
\newblock A reference case for mean field games models.
\newblock {\em J. Math. Pures Appl. (9)}, 92(3):276--294, 2009.

\bibitem{C2}
M.~Huang, P.~E. Caines, and R.~P. Malham{\'e}.
\newblock Large-population cost-coupled {LQG} problems with nonuniform agents:
  individual-mass behavior and decentralized {$\epsilon$}-{N}ash equilibria.
\newblock {\em IEEE Trans. Automat. Control}, 52(9):1560--1571, 2007.

\bibitem{C1}
M.~Huang, R.~P. Malham{\'e}, and P.~E. Caines.
\newblock Large population stochastic dynamic games: closed-loop
  {M}c{K}ean-{V}lasov systems and the {N}ash certainty equivalence principle.
\newblock {\em Commun. Inf. Syst.}, 6(3):221--251, 2006.

\bibitem{lst}
A.~Lachapelle, J.~Salomon, and G.~Turinici.
\newblock Computation of mean field equilibria in economics.
\newblock {\em Mathematical Models and Methods in Applied Sciences},
  20(04):567--588, 2010.

\bibitem{ll1}
J.-M. Lasry and P.-L. Lions.
\newblock Jeux \`a champ moyen. {I}. {L}e cas stationnaire.
\newblock {\em C. R. Math. Acad. Sci. Paris}, 343(9):619--625, 2006.

\bibitem{ll2}
J.-M. Lasry and P.-L. Lions.
\newblock Jeux \`a champ moyen. {II}. {H}orizon fini et contr\^ole optimal.
\newblock {\em C. R. Math. Acad. Sci. Paris}, 343(10):679--684, 2006.

\bibitem{ll3}
J.-M. Lasry and P.-L. Lions.
\newblock Mean field games.
\newblock {\em Jpn. J. Math.}, 2(1):229--260, 2007.

\bibitem{ll4}
J.-M. Lasry and P.-L. Lions.
\newblock Mean field games.
\newblock {\em Cahiers de la Chaire Finance et D\'eveloppement Durable}, 2007.

\bibitem{llg1}
J.-M. Lasry, P.-L. Lions, and O.~Gu{\'e}ant.
\newblock Application of mean field games to growth theory.
\newblock {\em Preprint}, 2010.

\bibitem{llg2}
J.-M. Lasry, P.-L. Lions, and O.~Gu{\'e}ant.
\newblock Mean field games and applications.
\newblock {\em Paris-Princeton lectures on Mathematical Finance}, 2010.

\bibitem{LCDF}
P.-L. Lions.
\newblock College de france course on mean-field games.
\newblock 2007-2011.

\bibitem{LIMA}
P.-L. Lions.
\newblock {I}{M}{A}, {U}niversity of {M}inessota. {C}ourse on mean-field games.
  {V}ideo. http://www.ima.umn.edu/2012-2013/sw11.12-13.12/.
\newblock 2012.

\bibitem{NguyenHuang}
S.~L. Nguyen and M.~Huang.
\newblock Linear-quadratic-{G}aussian mixed games with continuum-parametrized
  minor players.
\newblock {\em SIAM J. Control Optim.}, 50(5):2907--2937, 2012.

\bibitem{PIM}
E.A. Pimentel.
\newblock {\em Time dependent mean-field games.}
\newblock IST-UL. Doctoral thesis, Lisbon, 2013.

\bibitem{porretta}
A.~Porretta.
\newblock On the planning problem for the mean-field games system.
\newblock {\em Dyn. Games Appl.}, 2013.

\bibitem{Porb}
A.~Porretta.
\newblock {\em Weak solutions to Fokker-Planck equations and mean field games}.
\newblock Preprint, 2013.

\end{thebibliography}

\end{document}